\documentclass[a4paper]{article}
\usepackage[dvipdfmx]{graphicx}
\usepackage{amsmath, amssymb, amsfonts, amsthm}
\usepackage{ascmac}
\newtheorem{thm}{Theorem}
\newtheorem*{main}{\rm\bf Theorem \ref{Mainthm}}

\newtheorem{lem}{Lemma}

\newtheorem{prop}{Proposition}[section]

\begin{document}
\def\bslash{
\setlength{\unitlength}{1pt}
\thinlines 
\begin{picture}(10,10)
\put(-9.5,10){\line(2,-1){29}}
\end{picture}
}
\title{A sequence of $\pi /3$-equiangular hyperbolic polyhedra}

\author{Jun Nonaka}

\date{}

\maketitle
\begin{abstract}
Atkinson \cite{ref:atk} found a sequence of three-dimensional hyperbolic polyhedra whose dihedral angles are $\pi /3$. 
In this paper, we construct another sequence of such polyhedra. 
We also determine the volumes of some of these polyhedra. 
\end{abstract}

\textbf{Key words.} Coxeter polyhedron, hyperbolic 3-space, combinatorics. \\

\textbf{MSC (2020)}. 51F15, 51M20, 52B05\\


\section{Introduction}
We consider convex polyhedra in the hyperbolic 3-space $\mathbb{H}^3$ with finite volume. 
If a vertex of a polyhedron is not lying in the boundary $\partial \mathbb{H}^3$ of $\mathbb{H}^3$ but in $\mathbb{H}^3$, then it is {\it proper}. We call a vertex {\it ideal} if it belongs to $\partial \mathbb{H}^3$. 
If all vertices of a polyhedron are ideal, then this polyhedron is also called {\it ideal}. 
Since the volume of a polyhedron is finite, the number of ideal vertices of it is also finite. 
A polyhedron is called {\it simple} if any of its vertices belong to only three faces, and {\it simple at edges} if any of its edges belongs only to two faces. We call a polyhedron {\it almost simple} if it is simple at edges and any of its proper vertices belongs only to three faces. 
A Coxeter polyhedron in $\mathbb{H}^3$ is a convex polyhedron in $\mathbb{H}^3$ with dihedral angles of the form $\pi /m$ ($m\in \mathbb{N}$, $m\geq 2$) at all edges. 
A convex polyhedron is called {\it non-obtuse} if all the dihedral angles are less than or equal to $\pi /2$. Thus Coxeter polyhedra are non-obtuse polyhedra.  
If all dihedral angles of a Coxeter polyhedron are equal to $\pi /k$, then it is called {\it a} $\pi /k$-{\it equiangular polyhedron}. 
Andreev's theorem \cite{ref:and, ref:roe} says that if three faces of a non-obtuse polyhedron $P$ share a vertex of $P$, then the sum of the dihedral angles of them is at least $\pi $, in particular, the sum of them is $\pi $ if and only if the vertex is ideal. 
This theorem also says that if four faces of it share a vertex, then all the dihedral angles between them are equal to $\pi /2$. 
Thus, there is no $\pi /k$-equiangular polyhedron in $\mathbb{H}^3$ for $k>4$. 
In addition, if a $\pi /3$-equiangular polyhedron exists in $\mathbb{H}^3$, then it must be ideal since $\pi /3+\pi /3+\pi /3=\pi$. 
Let us consider the upper half-space model of $\mathbb{H}^3$ (see \cite[Chapter 4]{ref:r}, for example). 
In this model, the boundary of $\mathbb{H}^3$ consists of the point at infinity $\infty $ of the upper half-space and the Euclidean plane. 
The hyperplane of this model is the intersection of either a Euclidean sphere and $\mathbb{H}^3$ or a Euclidean hyperplane and $\mathbb{H}^3$, and its closure is orthogonal to the boundary of $\mathbb{H}^3$ in the Euclidean sense. 
Now we let a vertex of an ideal $\pi /3$-equiangular polyhedron $P$ be a point at infinity $\infty $ of the upper half-space model, and assume that this vertex is shared by exactly $m$ faces of $P$. 
Then 
\[ 2\pi =m(\pi -\pi /k) \geq m\pi /2.\]
Thus $m=3,4$. Then we obtain $(m,k)=(3,3), (4,2)$. 
This means that a $\pi /k$-equiangular polyhedron is not ideal if and only if $k=2$. 
By the above consideration, we also know that if a $\pi /k$-equiangular polyhedron is ideal, then $k$ is either 2 or 3. 

We remark that $\pi /2$-equiangular polyhedra are also called {\it right-angled}. 
There are a lot of results of the existence and examples of such polyhera. 
A compact $\pi /2$-equiangular polyhedron (i.e., the $\pi /2$-equiangular polyhedron has only proper vertices) in $\mathbb{H}^n$ exists only for $n\leq 4$ \cite{ref:vi}. One of the examples of compact $\pi /2$-equiangular polyhedra in $\mathbb{H}^4$ is known as 120-cell. 
Potyagailo, Vinberg and Dufour \cite{ref:du} proved that there are no $\pi /2$-equiangular polyhedra with finite volume in $\mathbb{H}^n$ for $n\geq 13$. However, examples of $\pi /2$-equiangular polyhedra of finite volume in $\mathbb{H}^n$ are known only for $n\leq 8$. 
In \cite{ref:kol}, Kolpakov proved that ideal $\pi /2$-equiangular polyhedra do not exist in $\mathbb{H}^n$ for $n\geq 7$. 
To obtain this result, he showed a polyhedron which has both a minimal volume and a minimal facet number among all ideal $\pi /2$-equiangular polyhedra in $\mathbb{H}^4$. 

In $\mathbb{H}^3$, there are also many results about $\pi /2$-equiangular polyhedra. 
Inoue show a sequence of disjoint unions of compact $\pi /2$-equiangular hyperbolic polyhedra in \cite{ref:In}. He also determined the first 825 polyhedra in the list of smallest volume compact $\pi /2$-equiangular hyperbolic polyhedra in \cite{ref:Ino}. 
We also see the table giving the initial segment in the list of volumes of compact $\pi /2$-equiangular polyhedra in \cite{ref:ves}. 
Vesnin and Egorov \cite{ref:ve} gave upper bounds on volumes of ideal $\pi /2$-equiangular polyhedra in terms the number of faces of the polyhedron. 

On the other hand, we do not know so many results of $\pi /3$-equiangular polyhedra as $\pi /2$-equiangular polyhedra. 
Atkinson estimates the volume of $\pi /3$-equiangular polyhedra in \cite{ref:atk}. 
In the same paper, he also introduced a sequence of $\pi /3$-equiangular polyhedra. 

In this paper, we give another sequence of $\pi /3$-equiangular polyhedra $\{ P_n \} $. 
Then we can see that such polyhedra exist infinitely many, and a certain amount of abundance. 
Moreover, we find the volumes of $P_n$. 
We also obtain the following theorem. 

\begin{main}
A polyhedron $P_2$ has the second smallest volume among all $\pi /3$-equiangular polyhedra in $\mathbb{H}^3$. 
\end{main}

We remark that the polyhedron which has minimal volume among all $\pi /3$-equiangular polyhedra is known as $P_1$ (see \cite{ref:atk}). 

\section{A sequence of $\pi$/3-angled polyhedra in $\mathbb{H}^3$}
The aim of this section is to give a sequence of $\pi /3$-equiangular polyhedra. 
For checking the existence of polyhedra in $\mathbb{H}^3$, the following theorem (called ``Andreev's theorem'') is useful 
(see also \cite{ref:roe}). 

\begin{thm}[\cite{ref:and}]\label{thm:And}
A non-obtuse almost simple polyhedron of finite volume with given dihedral angles, 
other than a tetrahedron or a triangular prism, exists in $\mathbb{H}^3$ 
if and only if the following conditions are satisfied$:$ 

$(a)$ if three faces meet at a proper vertex or an ideal vertex, then the sum of the dihedral angles 
between them is at least $\pi $ $( \pi $ for an ideal vertex$);$ 

$(b)$ if four faces meet at a proper vertex or an ideal vertex, then all the 
dihedral angles between them equal $\pi /2 ;$ 

$(c)$ if three faces are pairwise adjacent but share neither a proper vertex nor an ideal vertex, 
then the sum of the dihedral angles between them is less than $\pi ;$ 

$(d)$ if a face $F_i$ is adjacent to faces $F_j$ and $F_k$, 
while $F_j$ and $F_k$ are not adjacent but have a common ideal vertex which $F_i$ 
does not share, then at least one of the angles formed by 
$F_i$ with $F_j$ and with $F_k$ is different from $\pi /2 ;$ 

$(e)$ if four faces are cyclically adjacent but meet at neither a proper vertex nor an ideal vertex, 
then at least one of the dihedral angles between them is different from $\pi /2$. 
\end{thm}

If three faces of a polyhedron are pairwise adjacent but share neither a proper vertex nor an ideal vertex, we call these trio faces are $3$-circuit. 
By Theorem \ref{thm:And} (c), we know that a $\pi /3$-equiangular polyhedron except for a triangular prism, does not have a $3$-circuit. 
However, a $\pi /3$-equiangular triangular prism does not exist in $\mathbb{H}^3$. We will mention this in Section 3. 
Thus there are no trio faces that are adjacent to each other without sharing a vertex. 

From now on, we show an example of the sequence $\{ P_n\} $ of $\pi /3$-equiangular polyhedra in $\mathbb{H}^3$. 
Fix an ideal vertex denoted by $v_0$. Then there are exactly three faces that share $v_0$. Denote these three faces by $A_1$, $A_2$ and  $A_3$. We assume that $A_k$ $(k=1,2,3)$ are $(n+2)$-gons. We also assume that the face which is adjacent to both $A_1$ and $A_2$ (resp. $A_2$ and $A_3$, $A_3$ and $A_1$) other than $A_3$ (resp. $A_1$, $A_2$) is a quadrilateral and denote it by $B_1$ (resp. $B_2$, $B_3$). 
Moreover, we assume that any faces which are adjacent only to one of the faces $A_i$ and $B_i$ ($i=1,2,3$) are pentagons. Note that the number of these faces is $3(n-2)$ for $n\geq 3$. 
Finally, we assume that the faces which are not adjacent to any $A_i$ $(i=1,2,3)$, are hexagons. We also remark that the number of these hexagons is $(n-3)(n-2)/2$ for $n\geq 3$. 
By Theorem \ref{thm:And}, the above combinatorial structure can be realized a $\pi /3$-equiangular polyhedron. Then we denote this polyhedron by $P_n$. 
Note that the number of faces (resp. vertices) of $P_n$ is $(n^2+n+6)/2$ (resp. $n^2+n+2$). 
Figs. 1, 3, 4, 5 and 6 show the combinatorial structures of $P_n$. 
We remark that the polyhedra $P_1$ is a regular tetrahedron, and $P_2$ is a regular cube. 
To image these polyhedra easily, the upper half-space model of $\mathbb{H}^3$ may be useful. 
Fig. 2 shows the image of the combinatorial structure of $P_1$ with the upper half-space model of $\mathbb{H}^3$. We set that $v_0$ as the point at infinity of the upper half-space and we look at this polyhedron from $v_0$ in this figure. Since $v_0$ is the point at infinity, $A_i$ $(i=1,2,3)$ are parts of vertical Euclidean planes in this model. Then any face other than $A_i$ is represented by an upper hemisphere which orthogonally intersects with $\partial \mathbb{H}^3$ as Euclidean sense. The polyhedron $P_1$ is a tetrahedron, then a face that is represented by an upper hemisphere is unique. We can easily confirm that any dihedral angles of it are $\pi /3$. 
We also remark that the dotted circles in Figs. 1, 3, 4, and 5 show the intersection of $\partial \mathbb{H}^3$ and the upper hemispheres which consist of faces of $P_n$. 

\begin{figure}[htb]
\begin{center}
\includegraphics[width=80mm]{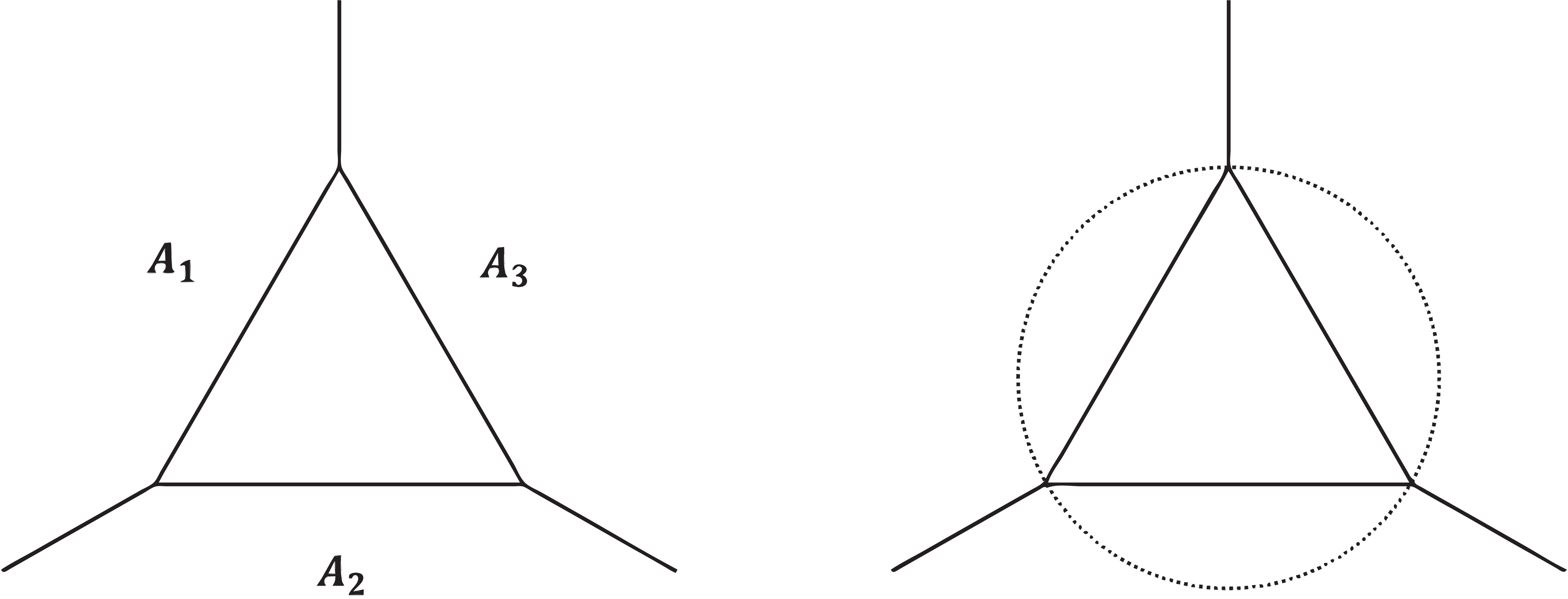}
\end{center}
\caption{The combinatorial structures of $P_1$}
\end{figure}

\begin{figure}[htb]
\begin{center}
\includegraphics[width=80mm]{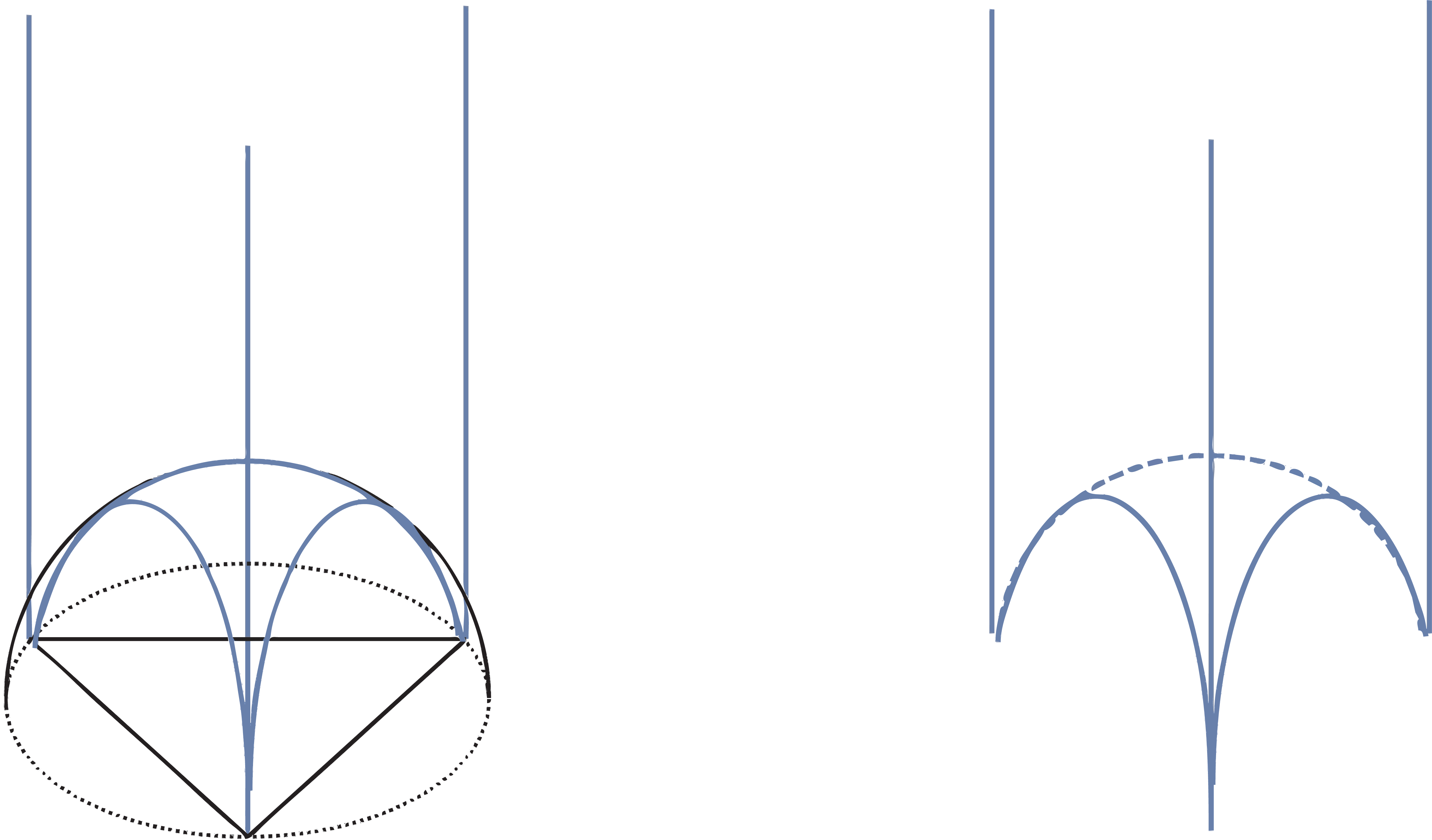}
\end{center}
\caption{Image of $P_1$ in upper half model of $\mathbb{H}^3$}
\end{figure}
\newpage 

\begin{figure}[htb]
\begin{center}
\includegraphics[width=110mm]{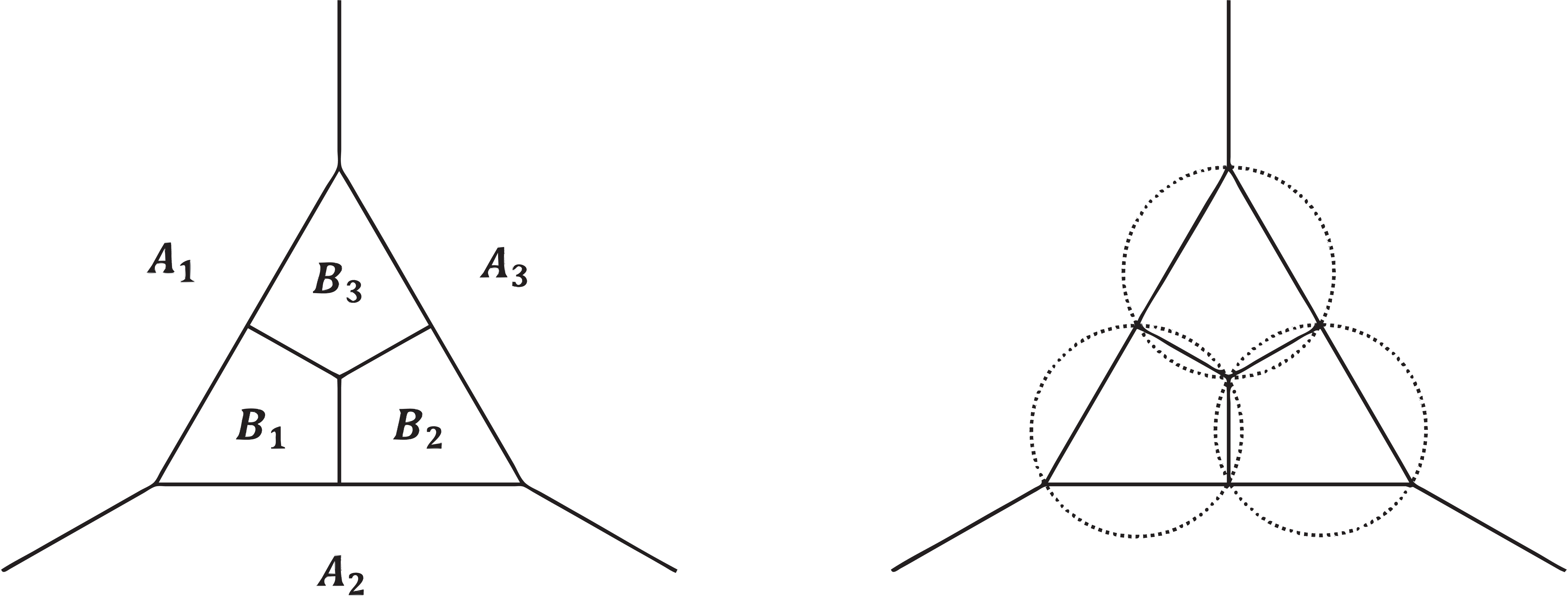}
\end{center}
\caption{The combinatorial structures of $P_2$}
\end{figure}

\begin{figure}[htb]
\begin{center}
\includegraphics[width=100mm]{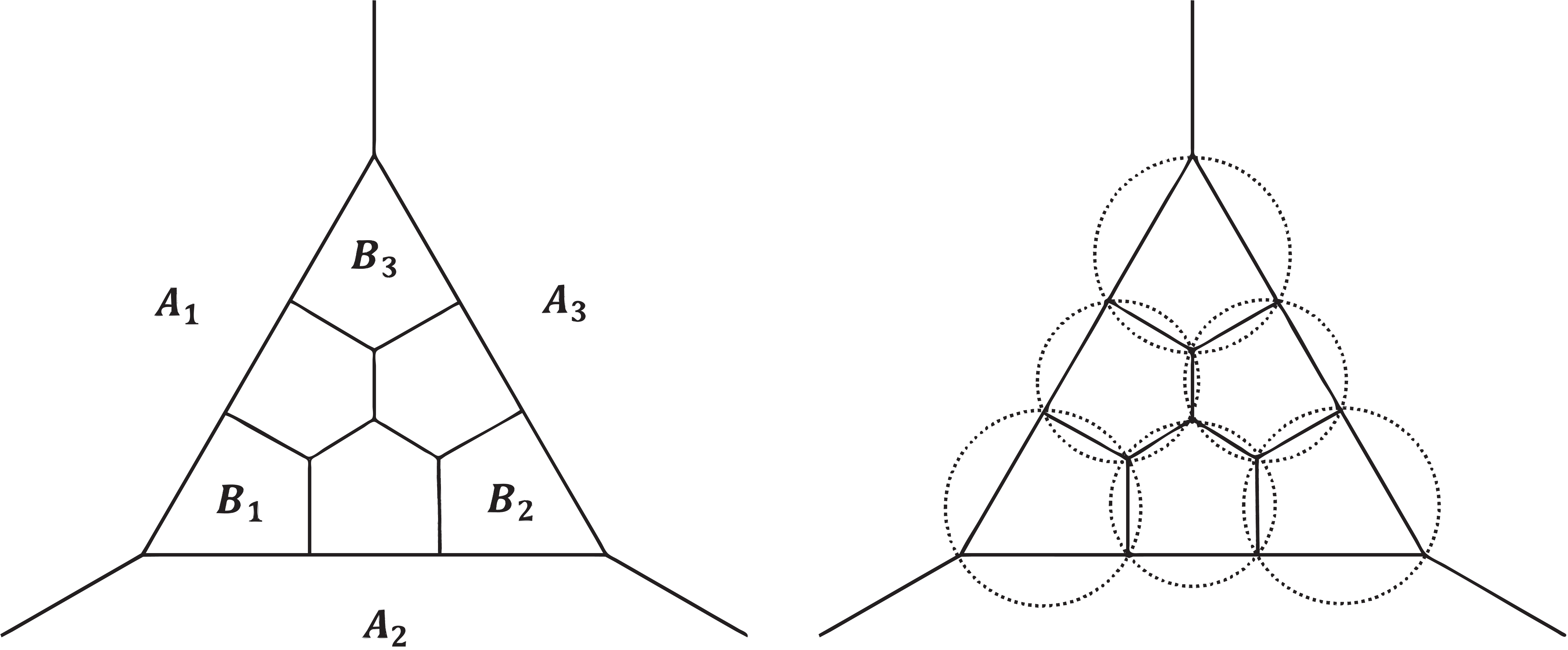}
\end{center}
\caption{The combinatorial structures of $P_3$}
\end{figure}

\begin{figure}[htb]
\begin{center}
\includegraphics[width=120mm]{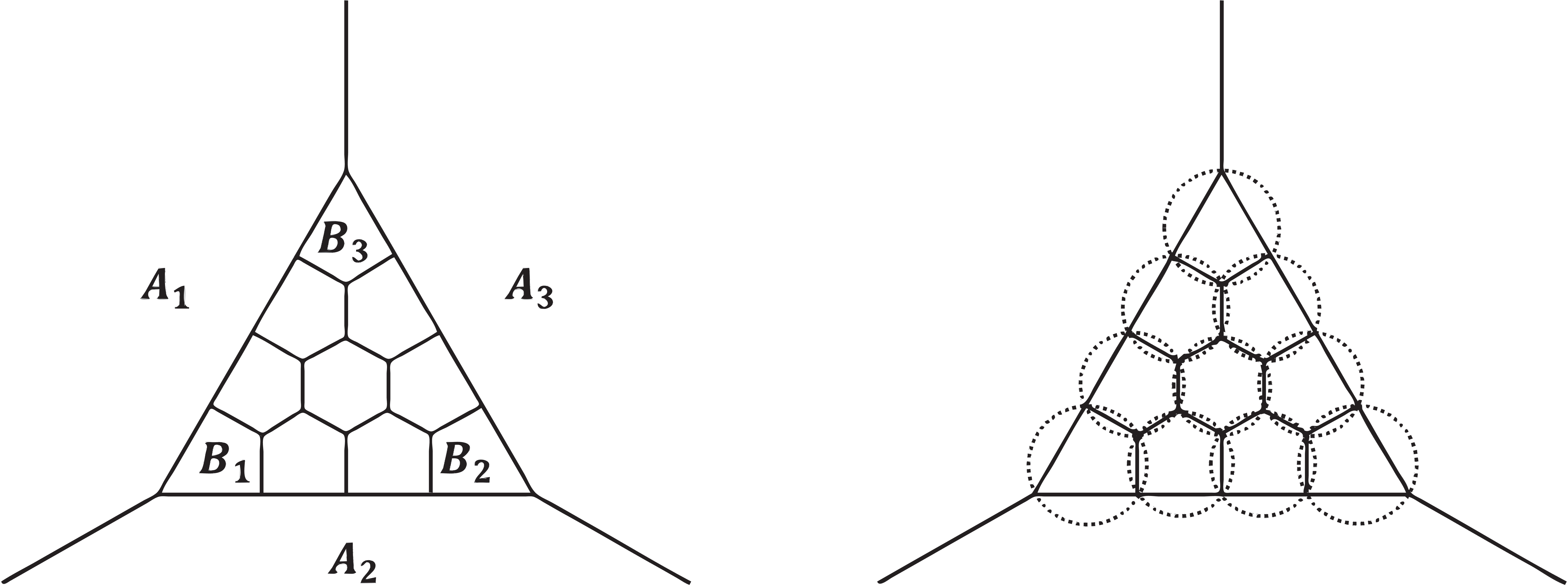}
\end{center}
\caption{The combinatorial structures of $P_4$}
\end{figure}
\newpage 

\begin{figure}[htb]
\begin{center}
\includegraphics[width=80mm]{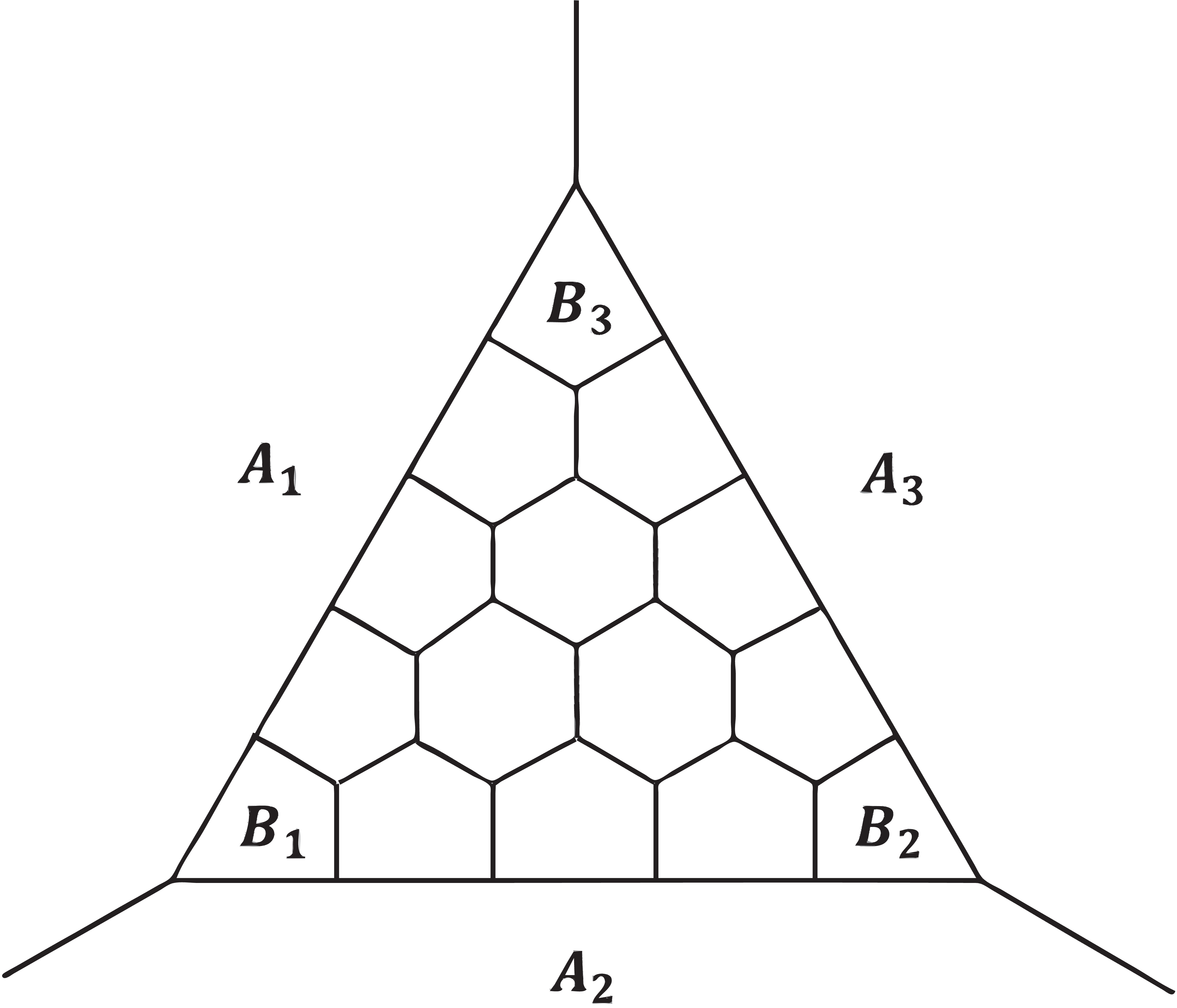}
\end{center}
\caption{The combinatorial structures of $P_5$}
\end{figure}

By this sequence of polyhedra, we see that there are infinitely many $\pi /3$-equiangular polyhedron in $\mathbb{H}^3$. 
We remark that there are also many $\pi /3$-equiangular polyhedron in $\mathbb{H}^3$ which depend on neither the sequence $\{ P_n\} $ 
nor the sequence that Atkinson introduced in \cite{ref:atk}. 
In the next section, we show an example of such a polyhedron. 

\section{Proof of Main Theorem}
In this section, we focus on the volumes of $\pi /3$-equiangular polyhedra in $\mathbb{H}^3$. 
Atkinson determined that the polyhedron which has the minimal volume among all $\pi /3$-equiangular hyperbolic polyhedra as follows. 

\begin{thm}[\cite{ref:atk}]\label{thm:Atk}
A three-dimensional $\pi /3$-equiangular hyperbolic polyhedron of minimal volume is $P_1$, up to isometry. 
\end{thm}

This theorem can be proved as follows. 

\begin{proof}
Let $P$ be a $\pi /3$-equiangular polyhedron other than $P_1$. 
Since $P$ is ideal, we can set one of the vertices of $P$ at infinity of the upper half-space model. 
We denote this vertex by $v_0$. The vertex $v_0$ is shared by exactly three faces. Denote these three faces by $A_1$, $A_2$ and $A_3$. 
These three faces are represented by planes perpendicular to the boundary of $\mathbb{H}^3$. 
We also denote by $\langle H\rangle $ the upper hemisphere which crosses to these three faces with its dihedral angles are $\pi /3$. 
Any face of $P$ other than $A_1$, $A_2$ and $A_3$ is represented by a part of the hemisphere. 
This hemisphere is smaller than the hemisphere $\langle H\rangle $. 
Moreover, in Euclidean sense, the distance between $v_0$ and this hemisphere is larger than the distance between $v_0$ and $\langle H\rangle $. 
Thereby $P$ can be decomposed into two parts by $\langle H\rangle $. 
The part which has $v_0$ is $P_1$. 
Thus, the volume of $P$ is larger than that of $P_1$. 
\end{proof}

A lot of volume formulae of hyperbolic polyhedra are known. In this paper, we use the following formula for the volume of tetrahedra in $\mathbb{H}^3$ introduced by Milnor. 

\newfam\cyrfam
\font\tencyr=wncyr10
\def\cyr{\fam\cyrfam\tencyr}
\def\russianL{\mathop{\hbox{\cyr L}}}
\begin{prop}[\cite{ref:mil}]
Let $T(\alpha ,\beta , \gamma )$ be an ideal tetrahedron in $\mathbb{H}^3$ with its dihedral angles are $\alpha $, $\beta $ and $\gamma $
Then its volume is 
\begin{align*}
\hbox{vol}(T(\alpha ,\beta ,\gamma))={\russianL}(\alpha )+{\russianL}(\beta )+{\russianL}(\gamma )
\end{align*}
where 
\begin{align*}
{\russianL}(x)=
\frac{1}{2}\,\sum_{r=1}^{\infty}\frac{\sin(2rx)}{r^2}
=-\int\limits_0^{x}\,\log\mid2\sin t\mid dt\,,\,\,x\in\mathbb R\,,
\end{align*}
is Lobachevsky's function. 
\end{prop}

By this formula, we obtain the volume 
\begin{align*}
vol(P_1)=vol(T(\pi /3,\pi /3,\pi /3))=3{\russianL}(\pi /3)\sim 1.014941. 
\end{align*}
From now on, we prepare some lemmas for comparing some volumes of $\pi /3$-equiangular polyhedra. 

\begin{lem}\label{lem:1} 
If an ideal tetrahedron in $\mathbb{H}^3$ has at least three dihedral angles which are $\pi /3$, then any dihedral angles of it are $\pi /3$.  
\end{lem}

\begin{proof}
By the assumption, the tetrahedron has a trio of faces at least two of three dihedral angles are $\pi /3$. 
Since the sum of the dihedral angles of these three faces is $\pi $, the rest dihedral angle is also $\pi /3$. 
Now we let the ideal vertex shared by those three faces be at infinity of upper half space. 
Then those three faces are included in planes that are perpendicular to the boundary of $\mathbb{H}^3$ in the Euclidean sense. 
In this case, the face which does not have the vertex at infinite of upper half-space must be a unique upper hemisphere, 
and then the rest three dihedral angles determine $\pi /3$. 
\end{proof}

We assume that a $\pi /3$-equiangular polyhedron is a triangular prism. Then we can describe the combinatorial structure as in Fig. 7. 
This prism consists of the faces $A_1$, $A_2$ and $A_3$ which share vertex $v_0$ which is the point at infinity, and the faces $B$ and $B'$ which are the parts of the upper hemispheres. 
Note that the bases of this prism are $A_2$ and $B$. 
By Lemma \ref{lem:1}, a tetrahedron whose vertices are $v_0$, $v_3$, $v_4$ and $v_5$ is a $\pi /3$-equiangular polyhedron. 
This tetrahedron is obtained by decomposing the triangular prism by the halfplane which passes three vertices $v_0$, $v_4$ and $v_5$. This means that the dihedral angle between the face $B$ and the face $B'$ is greater than $\pi /3$. Thus, there is no $\pi /3$-equiangular triangular prism. 

\begin{figure}[htb]
\begin{center}
\includegraphics[width=50mm]{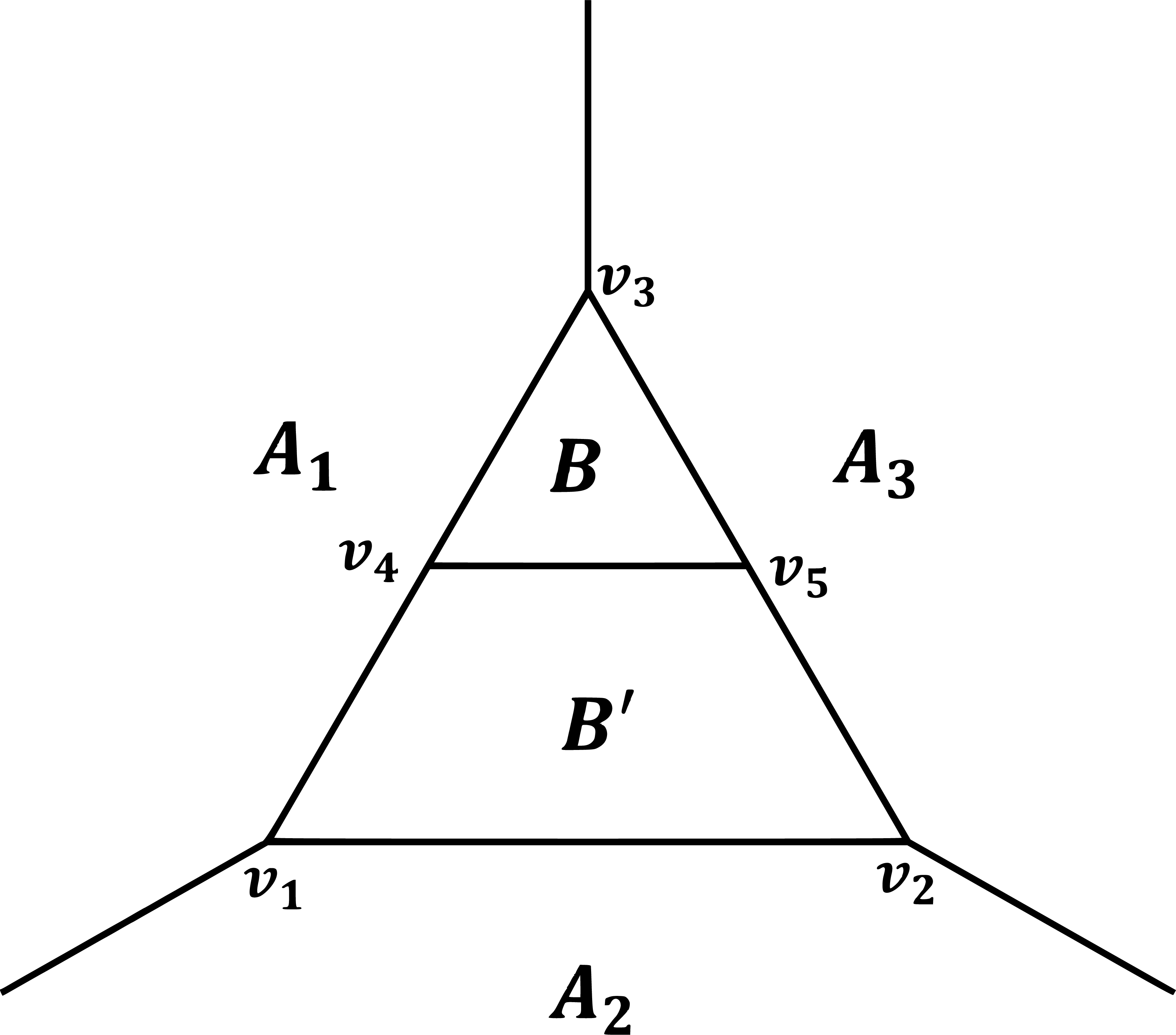}
\end{center}
\caption{The combinatorial structures of a triangular prism}
\end{figure}

From now on, we denote by $A_i$ and $B_i$ $(i=1,2,3)$ the faces of $P_2$ and by $v_0$ the vertex of $P_2$ as before. 
We also denote the vertices of $P_2$ by $v_k$ $(k=1,2,\cdots ,7)$ as in Fig. 3. 
Then, we can break down $P_2$ into five tetrahedra: $v_0v_1v_2v_3$, $v_1v_2v_5v_7$, $v_2v_3v_6v_7$, $v_1v_3v_4v_7$, $v_1v_2v_3v_7$. 
We can confirm that each of these five tetrahedra has at least three dihedral angles which are $\pi /3$. 
Thus these are isometry to $P_1$ by Lemma \ref{lem:1}. 
This means that the volume $vol(P_1)$ of $P_1$ and that $vol(P_2)$ of $P_2$ have the relation; 
\[ vol(P_2)=5vol(P_1). \]

Now we add one more Lemma giving  a sufficient condition that a polyhedron is an isometry to $P_1$. 
\begin{lem}\label{lem:2}
Let $P$ be a $\pi /3$-equiangular polyhedron in $\mathbb{H}^3$. 
If $P$ has a triangular face, then $P$ must be $P_1$. 
\end{lem}

\begin{proof}
Denote by $A_1$, $A_2$, and $A_3$ the faces which share a vertex of $P$. We also denote by $v_0$ this vertex. 
Without a loss of generality, we can suppose that $A_1$ is a triangle. 
Then there is exactly one face that is adjacent to $A_1$ neither $A_2$ nor $A_3$. 
Denote this face by $B_1$. We note that $B_1$ is also adjacent both to $A_2$ and $A_3$ since $A_1$ is a triangle. 
We also note that $B_1$ does not have $v_0$. 
It is obvious that $P$ is $P_1$ when both $A_2$ and $A_3$ are triangles. 
Thus, we need to consider the case that at least one of $A_2$ and $A_3$ is not a triangle. 

We suppose that $A_2$ is not a triangle. 
Then the endpoint of the edge $A_2\cap A_3$ other than $v_0$ cannot be shared by $B_1$. 
This endpoint is shared by $A_2$, $A_3$, and a face is neither $A_1$ nor $B_1$.  
This means that $A_2$, $A_3$, and $B_1$ do not share an ideal vertex, and then there is a $3$-circuit. 
A contradiction. Thus, $A_2$ is a triangle. 
In the same manner as before, we can show that $A_3$ is also a triangle. 

Hence $P$ must be $P_1$. 
\end{proof}

\begin{lem}
Let $P$ be a $\pi /3$-equiangular polyhedron in $\mathbb{H}^3$. 
If $P$ has three quadrilateral faces which share one ideal vertex, then $P$ must be $P_2$. 
\end{lem}

\begin{proof}
Denote by $A_1$, $A_2$ and $A_3$ the quadrilateral faces share one ideal vertex of $P$. We also denote by $v_0$ this vertex. 
Denote by $B_1$ (resp. $B_2$, $B_3$) the face which is adjacent both $A_1$ and $A_2$ (resp. $A_2$ and $A_3$, $A_3$ and $A_1$). 
Since $A_1$, $A_2$ and $A_3$ are quadrilaterals, then $B_1$, $B_2$, and $B_3$ are mutually adjacent. 
Then $B_1$, $B_2$, and $B_3$ share an ideal vertex of $P$ since $P$ does not have $3$-circuit. 
This means that the combinatorial structure of $P$ coincides with that of $P_2$. 
Hence $P$ is $P_2$. 
\end{proof}

Now we consider the combinatorial structure of a $\pi /3$-equiangular polyhedron other than $P_1$ and $P_2$. 
Let denote this polyhedron by $P$. 
The combinatorial structure of $P$ must satisfy the following rules; 

(1) any vertex is shared by exactly three faces

(2) there is no triangle

(3) there is no $3$-circuit

(4) any two faces can be adjacent by only one edge. \\ 

Let $v(P)$, $e(P)$ and $f(P)$ be the number of ideal vertices, edges and faces of $P$. 
Then by rule (1), we obtain 
\[ 2e(P)=3v(P).\] 
We also obtain Euler's identity
\[ v(P)-e(P)+f(P)=2.\] 
Then by these two identities, we obtain 
\[ f(P)=\frac{v(P)}{2}+2.\]
This identity shows that $v(P)$ is even. 
It is clear that $v(P)\geq 10$. 
We remark that by the rules (1)-(4) and the above identities, we may consider the case $v(P)=4$, $8$. 
But if $v(P)=4$ (resp. $v(P)=8$), then $P$ must be $P_1$ (resp. $P_2$). 
Note that we remove the case that $P$ is both of these polyhedra. 

From now on, we consider the case that $v(P)=10$. 
In this case, the combinatorial structure of $P$ must be as Fig. 8. 
We denote by $v_0$ an ideal vertex which is shared by three faces $A_1$, $A_2$ and $A_3$ as in Fig. 8. 
We also denote the faces by $B_i$ $(i=1,2,3)$ and the ideal vertices $v_j$ $(j=1,2,\cdots ,9)$ as in Fig. 8. 
Note that $P$ is a pentagonal prism in this case. 
We denote by $\langle H_1\rangle $ (resp. $\langle H_2\rangle $, $\langle H_3\rangle $ and $\langle H_4\rangle $), by the halfplane which passes through trio ideal vertices $(v_0,v_4,v_5)$ (resp. $(v_0,v_6,v_7)$, $(v_0,v_4,v_7)$ and $(v_4,v_5,v_9)$). 
Then, we decompose $P$ by these four halfplane $\langle H_k\rangle $ $(k=1,2,3,4)$, 
we obtain four tetrahedra $v_0v_1v_4v_5$, $v_0v_2v_6v_7$, $v_0v_3v_4v_7$, $v_4v_5v_8v_9$ and a rest part of $P$.  
By Lemma, these four tetrahedra are congruent to $P_1$. 
Let denote that the rest part of $P$ by $P'$. 
We can see the combinatorial structure of $P'$ in Fig. 9. 

\begin{figure}[htb]
\begin{center}
\includegraphics[width=60mm]{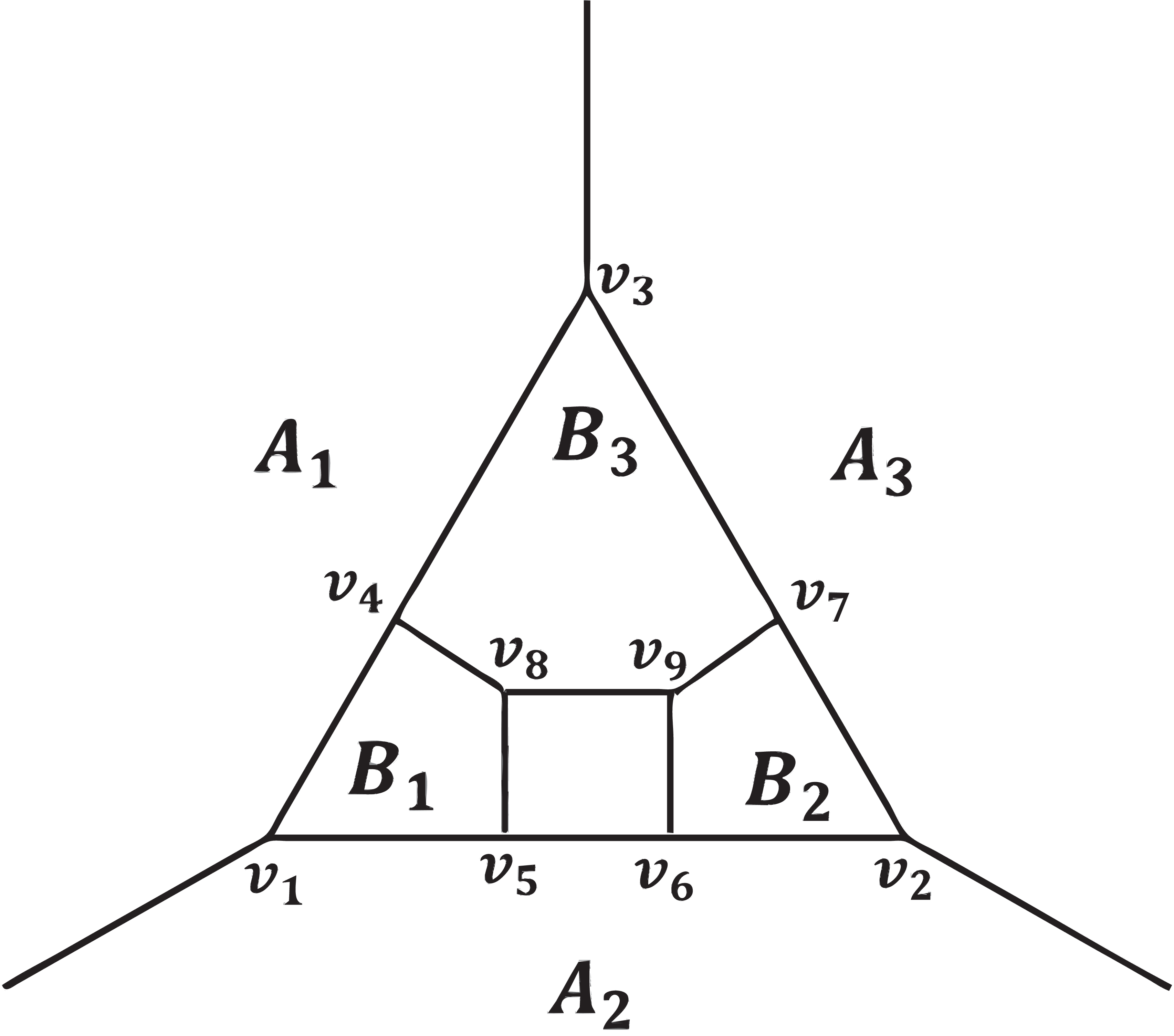}
\end{center}
\caption{The case $v(P)=10$}
\end{figure} 

\begin{figure}[htb]
\begin{center}
\includegraphics[width=80mm]{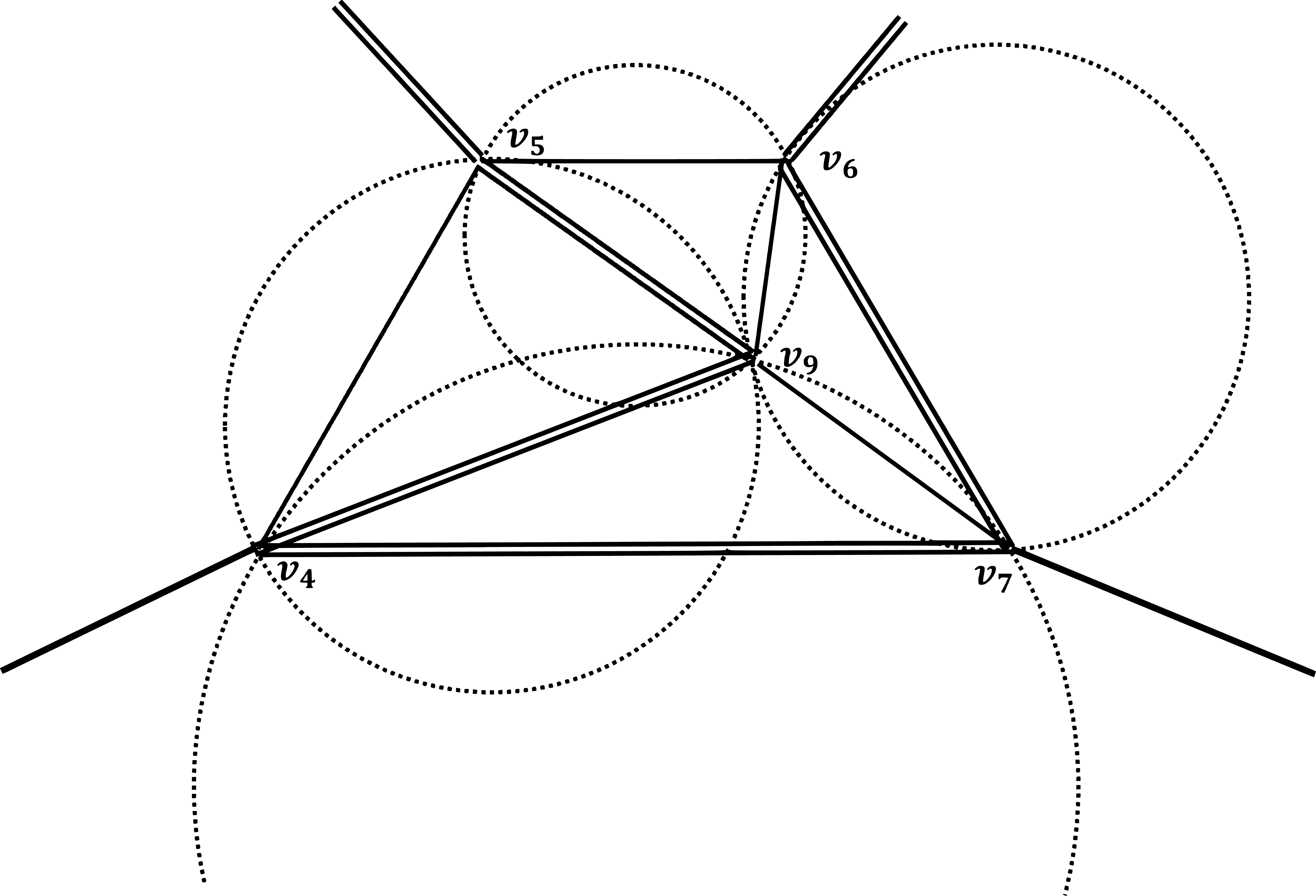}
\end{center}
\caption{The combinatorial structure of $P'$}
\end{figure} 

The doubled straight lines in Fig. 9 are represent the dihedral angles consist by these edges are $2\pi /3$, and the other dihedral angles are $\pi /3$. 
Now we decompose $P'$ by four tetrahedra: $v_0v_4v_5v_9$, $v_0v_4v_7v_9$, $v_0v_6v_7v_9$ and $v_0v_5v_6v_9$. 
By using some methods of fundamental Euclidean geometry, we find that two tetrahedra $v_0v_4v_5v_9$ and $v_0v_5v_6v_9$ are congruent and whose dihedral angles are consisted by $\pi /3$, $a$ and $b+\pi /3$. 
The values $a$ and $b$ are the measures of acute angles which satisfy $a=\arccos (\sqrt{10} /4)$ and $b=\arccos (\sqrt{2}(3+\sqrt{5})/8)$. 
Two tetrahedra $v_0v_4v_7v_9$ and $v_0v_6v_7v_9$ are also congruent. These two tetrahedra have dihedral angles $2\pi /3$, $a$ and $b$. 
Then, we obtain the volume of $P'$ 
\newfam\cyrfam
\font\tencyr=wncyr10
\def\cyr{\fam\cyrfam\tencyr}
\def\russianL{\mathop{\hbox{\cyr L}}}
\begin{align*}
vol(P')&= 2\left( {\russianL}(\pi /3)+{\russianL}(a)+{\russianL}(b+\pi /3)\right)+2\left( {\russianL}(2\pi /3)+{\russianL}(a)+{\russianL}(b)\right) \\
&\sim 3.137614. 
\end{align*}
Then we obtain the relationship between the volume of $P_2$ and that of $\pi /3$-equiangular pentagonal prism $P$ as follows; 
\begin{align*}
vol(P)= 4vol(P_1)+vol(P')\sim 7.198378> vol(P_2)=5vol(P_1)\sim 5.074705. 
\end{align*}

\begin{figure}[htb]
\begin{center}
\includegraphics[width=80mm]{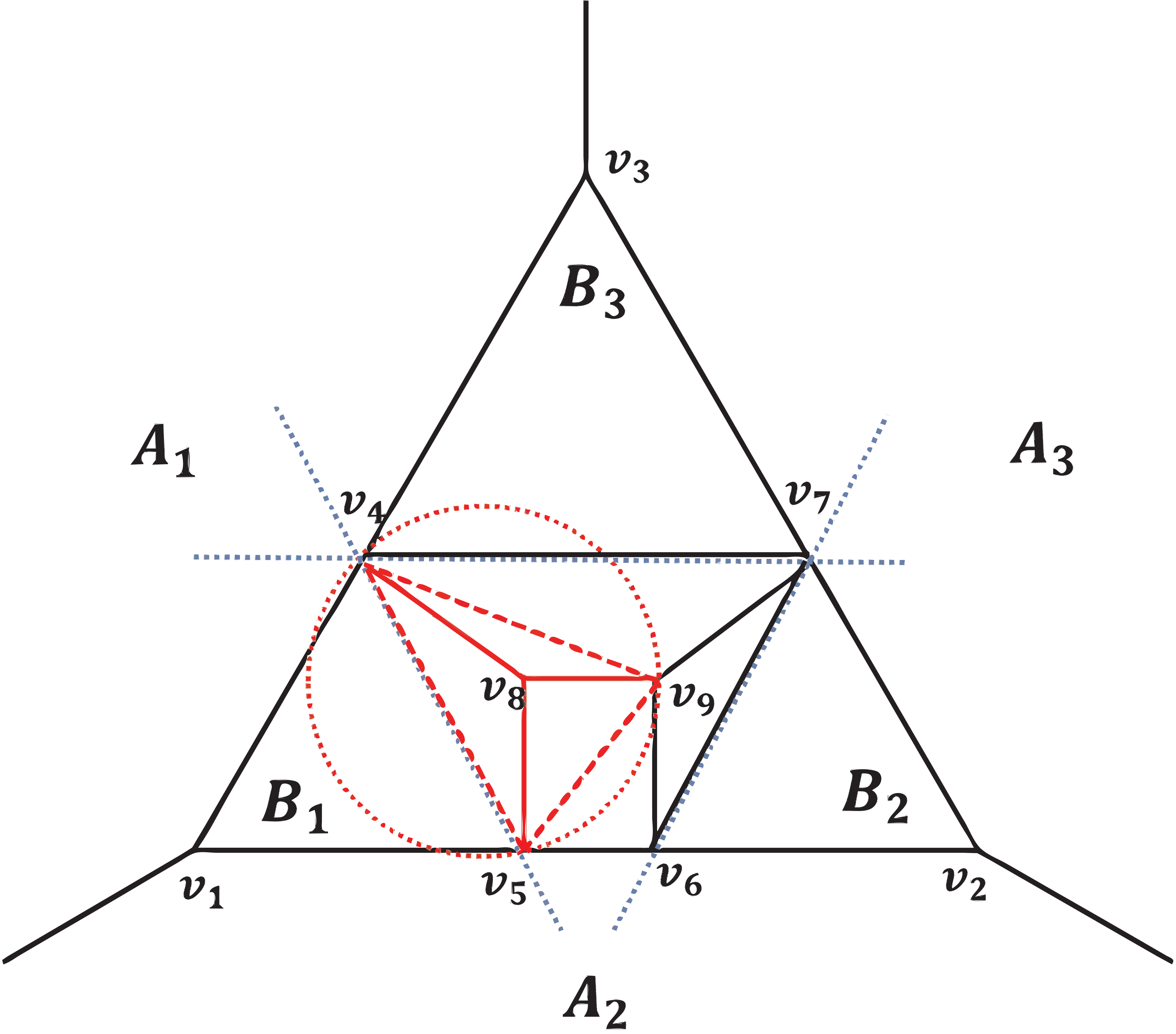}
\end{center}
\caption{The image of decomposition of $P$ which satisfies $v(P)=10$}
\end{figure}

Next, we consider the case that $P$ satisfies $v(P)\geq 12$. 
Denote by the ideal vertices $v_0$, $v_1$, $v_2$, $v_3$, the faces $A_1$, $A_2$, $A_3$, $B_1$, $B_2$ and $B_3$ as before. 
Since at least one of $A_1$, $A_2$ and $A_3$ are neither triangles nor quadrilaterals, we can set that $A_2$ has more than four sides. 
Then there is a face that is adjacent to $A_2$ other than $A_1$, $A_2$, $B_1$, and $B_2$. Let denote this face by $B_4$. 
By Lemma \ref{lem:2}, $B$ has at least two ideal vertices that are not shared by any of faces $A_1$, $A_2$ and $A_3$. 
Select such two vertices and denote them by $v_x$ and $v_y$. 

Consider the case that $v_x$ and $v_y$ are not connected by a side of $B_4$. 
Then by using a halfplane that passes through three vertices which connect to $v_x$ by the edges of $P$, 
we can decompose a tetrahedron that has these three vertices and a vertex $v_x$ from $P$. 
In the same manner, we can also decompose the other tetrahedron one of whose vertices is $v_y$ from $P$. 
Moreover, we can decompose three tetrahedra, each of those has the edge $v_0v_k$ for $k=1$, 2, 3. 
Thus, in this case, $vol(P)>5vol(P_1)=vol(P_2)$. 

Now we consider otherwise, i.e., $v_x$ and $v_y$ are connected by a side of $B_4$. 
If $B_4$ is not a quadrilateral, then we can select two vertices which are not shared by $A_2$ and are not connected by a side of $B_4$. 
In this case, we can show that $vol(P)>5vol(P_1)=vol(P_2)$ as above. 
So, we only need to consider the case that $B_4$ is a quadrilateral. 

We suppose that $P$ has exactly two vertices $v_x$ and $v_y$ that are not shared by any of faces $A_1$, $A_2$ and $A_3$. 
Then one of the three edges that share $v_x$ has an ideal vertex of either $A_1$ or $A_3$. 
One of three edges that share $v_y$ has also an ideal vertex of either $A_1$ or $A_3$. 
Thus, since $P$ has no $3$-circuit, then $B_4$ must be adjacent to four faces: $A_2$, $B_1$, $B_2$ and $B_3$.    
This means that the combinatorial structure of $P$ coincides with as Fig. 7, and it does not satisfy condition $v(P)\geq 12$. 
Thus there exist other ideal points that are not shared by any of faces $A_1$, $A_2$ and $A_3$. 
Denote this new ideal vertex by $v_z$. 
Since $P$ does not have a triangular face, $v_z$ is neither connected to $v_x$ nor $v_y$ by an edge of $P$. 
We can assume that $v_z$ is not connected to $v_x$ by an edge. 
Then, in the same way as above, by using a halfplane that passes through three vertices that connect to $v_x$ by the edges of $P$, we can decompose a tetrahedron that has these three vertices and a vertex $v_x$ from $P$. 
We also obtain another tetrahedron which has a vertex $v_z$ in the same way as above. 
In addition, we obtain three tetrahedra form $P$, each of those has the edge $v_0v_k$ for $k=1$, 2, 3 (see also Fig. 11). 

\begin{figure}[htb]
\begin{center}
\includegraphics[width=80mm]{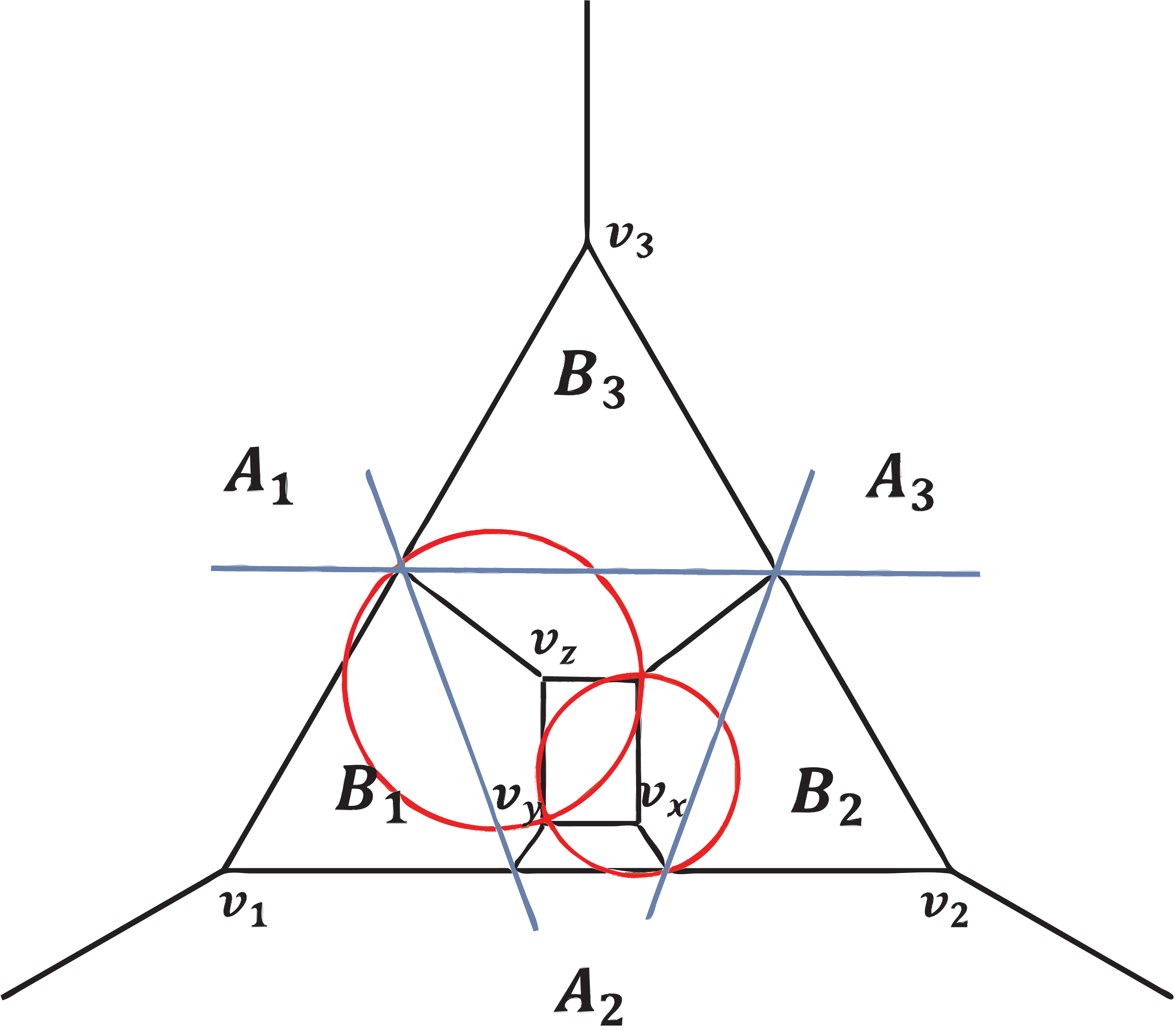}
\end{center}
\caption{An example of decomposition of $P$ which satisfies $v(P)=12$}
\end{figure}

Thus, we get $vol(P)>5vol(P_1)=vol(P_2)$. 

Hence, we proved the following Theorem. 

\begin{thm}\label{Mainthm}
The Polyhedron $P$ has the second smallest volume among all $\pi /3$-equiangular polyhedra in $\mathbb{H}^3$. 
\end{thm}

\section{The volume of $P_n$}
In the previous section, we showed that $P_2$ has the second smallest volume among all $\pi /3$-equiangular polyhedra. 
The aim of this section is to give the volume of $P_n$. 

By the same way of obtaining the relation of the volume 
\[ vol(P_2)=5vol(P_1) \]
As we mentioned in Section 2, we can decompose $P_n$ by $n(3n-1)/2$ tetrahedra which are isometry to $P_1$. 
Figs. 12--14 show the way of such decomposition of $P_n$ for $2\leq n\leq 5$. 
We also note that the dotted lines and the dotted circles show the hyperfaces for decomposing $P_n$ in these figures. 
Then we obtain
\[ vol(P_n)=\frac{n(3n-1)}{2}vol(P_1). \]

\begin{figure}[htb]
\begin{center}
\includegraphics[width=40mm]{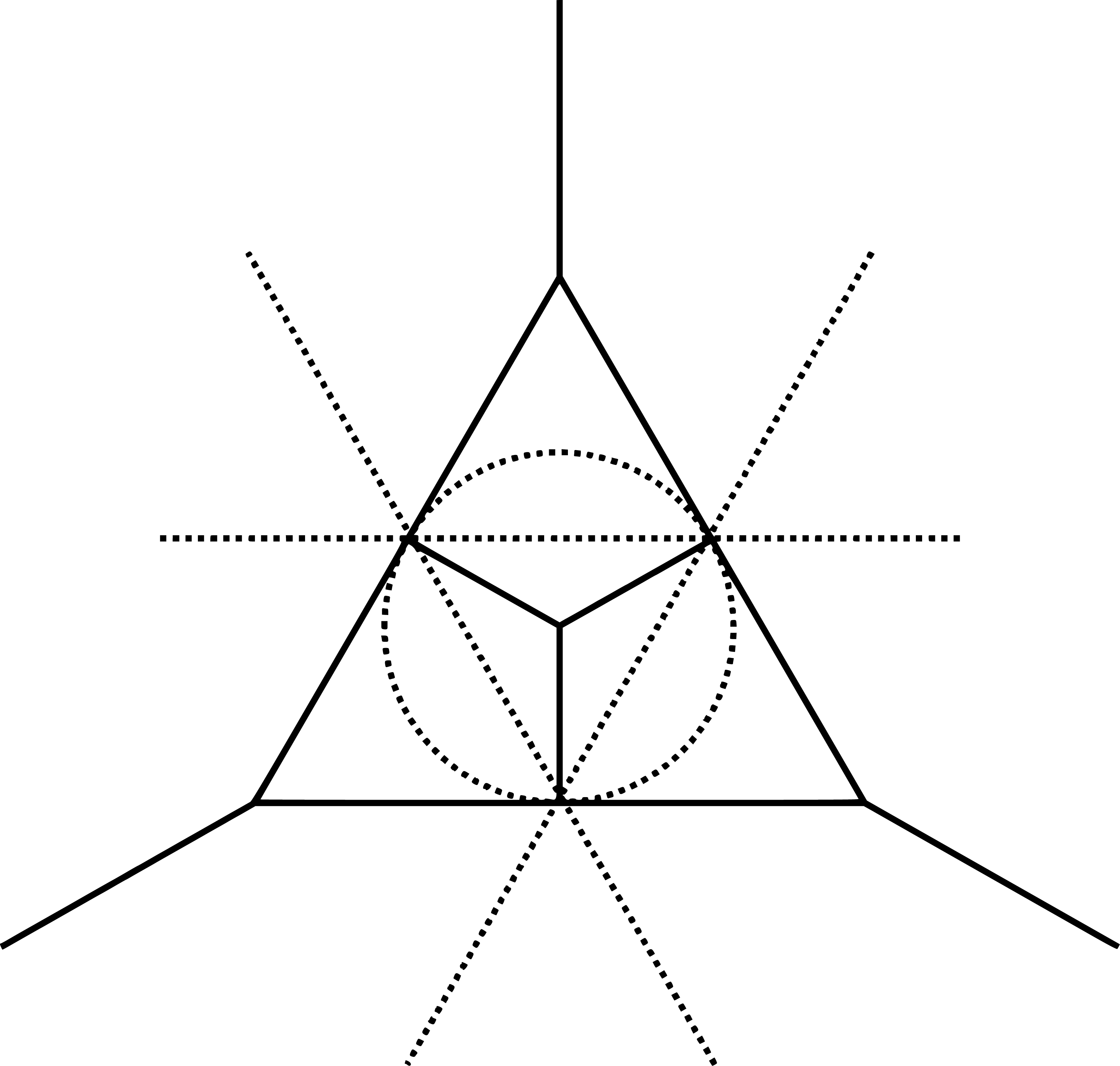}
\end{center}
\caption{Decomposition of $P_2$ by five tetrahedra}
\end{figure}

\begin{figure}[htb]
\begin{center}
\includegraphics[width=50mm]{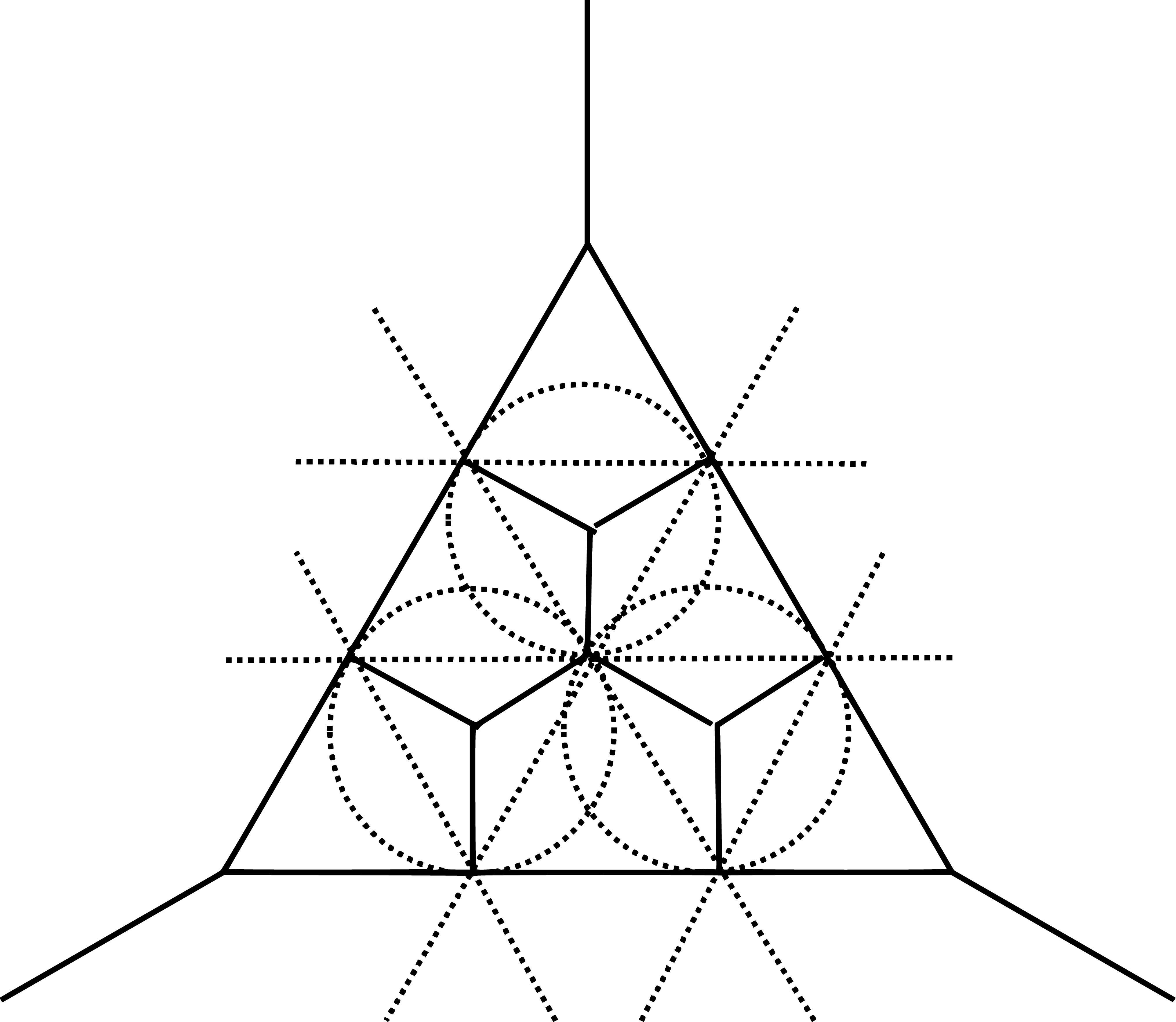}
\end{center}
\caption{Decomposition of $P_3$ by twelve tetrahera}
\end{figure}

\begin{figure}[htb]
\begin{center}
\includegraphics[width=60mm]{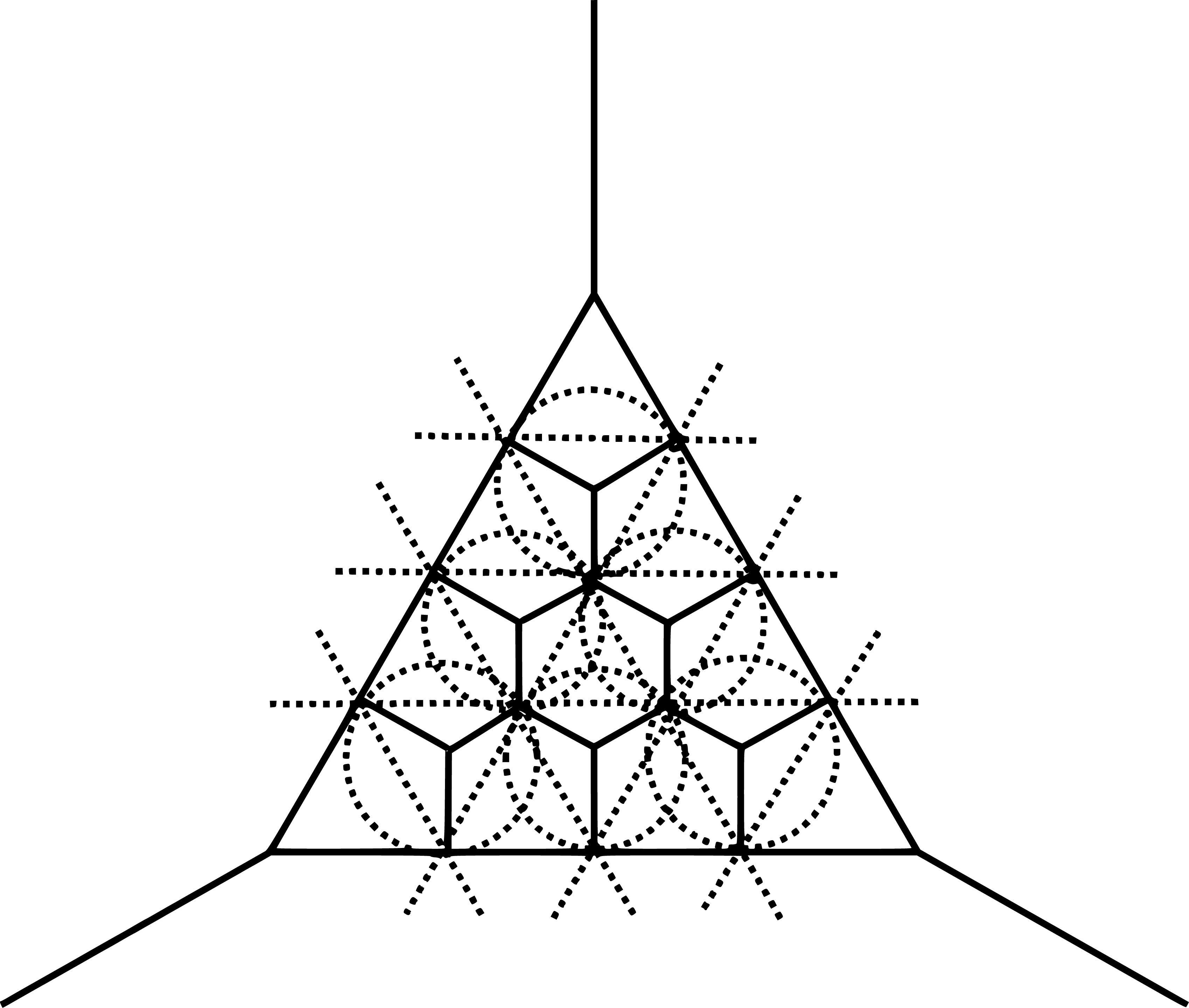}
\end{center}
\caption{Decomposition of $P_4$ by twenty two tetrahedra}
\end{figure}

\begin{figure}[htb]
\begin{center}
\includegraphics[width=60mm]{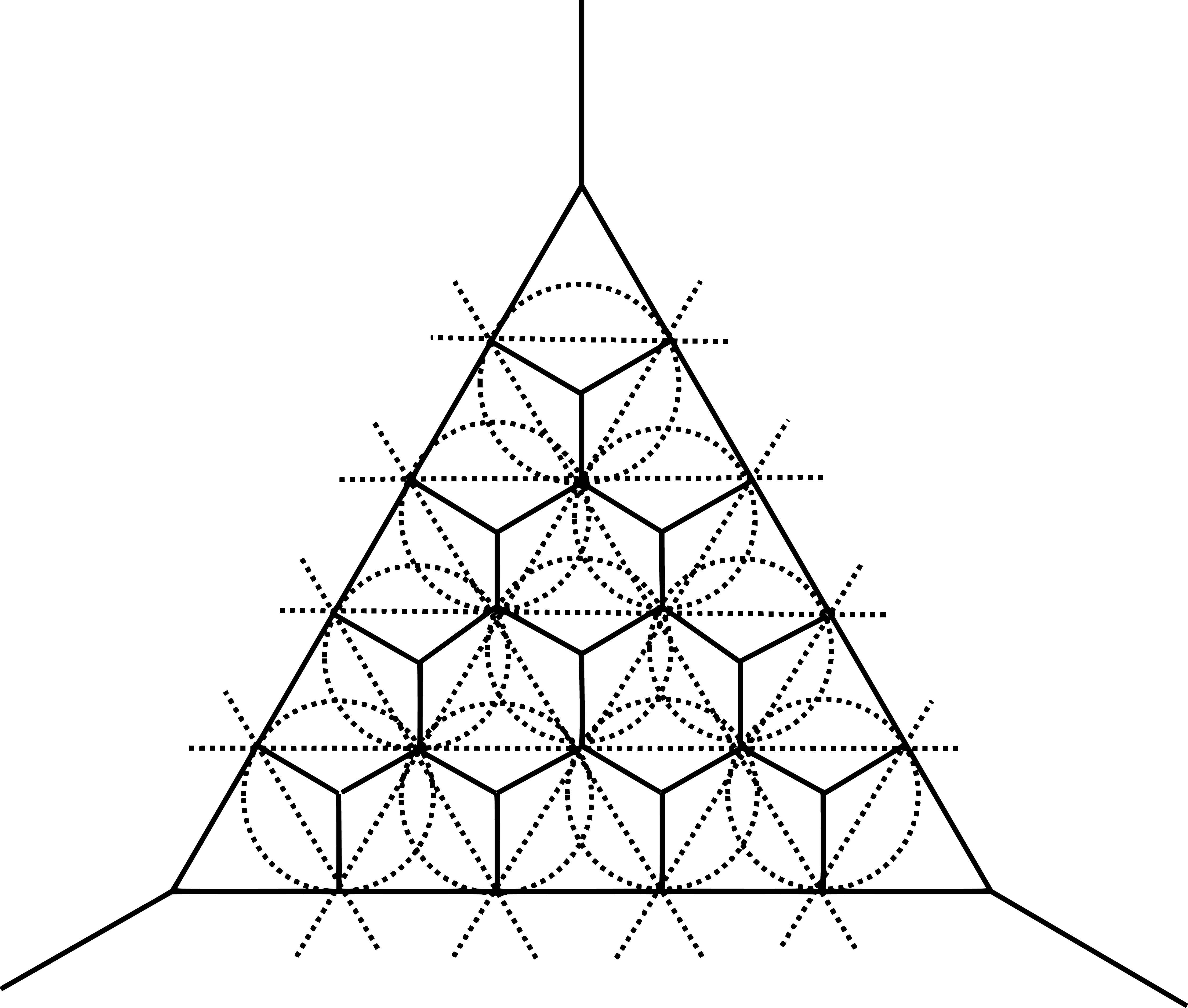}
\end{center}
\caption{Decomposition of $P_5$ by thirty five tetrahedra}
\end{figure}
\clearpage 

\subsection*{Acknowledgment}
The author would like to thank Tomoshige Yukita for telling him the reference \cite{ref:atk}. 
This paper is a part of outcome of research performed under a Waseda University Grant for Special Research Projects (Projects number: 2021C-351).



\vspace{1cm}
\begin{flushleft}
Waseda University Junior and Senior High School, \\
3-31-1 Kamishakujii, Nerima-ku, Tokyo, 177-0044, Japan. \\
email: \verb|jun-b-nonaka@waseda.jp|
\end{flushleft}
\end{document}